\documentclass[12pt]{article}
\usepackage[affil-it]{authblk}
\usepackage{amsmath} 
\usepackage{amssymb} 
\usepackage{graphicx} 
\usepackage{enumerate}
\usepackage[utf8]{inputenc}
\usepackage{amsthm}
\usepackage{color}

\newtheorem{thm}{Theorem}

\newcommand{\F}{\mathbb{F}}
\newcommand{\Fq}{\mathbb{F}_q}

\begin{document}

\title{Carlitz Rank and Index of Permutation Polynomials}

\date{}
\author{Leyla I\c s\i k$^1$, Arne Winterhof$^2$, }

\maketitle

\noindent
$^1$ Salzburg University, Hellbrunnerstr.\ 34,
5020 Salzburg, Austria\\
E-mail: leyla.isik@sbg.ac.at\\

\noindent
$^2$ Johann Radon Institute for Computational and Applied Mathematics\\
Austrian Academy of Sciences, Altenbergerstr.\ 69, 4040 Linz, Austria\\
E-mail: arne.winterhof@oeaw.ac.at

\begin{abstract} Carlitz rank and index are two important measures for the complexity of a permutation polynomial
$f(x)$ over the finite field $\F_q$. In particular, 
for cryptographic applications we need both, a high Carlitz rank and a high index. In this article we study the relationship between Carlitz rank $Crk(f)$ and index $Ind(f)$. 
More precisely, if the permutation polynomial is neither close to a polynomial of the form $ax$ nor a rational function of the form $ax^{-1}$, 
then we show that $Crk(f)>q- \max\{3 Ind(f),(3q)^{1/2}\}$. Moreover we show that the permutation polynomial which represents the discrete logarithm guarantees both a large index and a large Carlitz rank.

\end{abstract}

\bigskip


\bigskip


\section{Introduction}

In 1953, L. Carlitz \cite{Carlitz53} proved that all permutation polynomials over the finite field $\Fq$
of order $q\ge 3$ are compositions of linear polynomials $ax+b$, $a,b \in \Fq$, $a\neq0$, 
and inversions $x^{q-2}=\left\{\begin{array}{cc} 0, &x=0,\\ x^{-1}, & x\ne 0,\end{array}\right.$ see \cite{Carlitz53} or \cite[Theorem~7.18]{LidNid97}. 
Consequently, any permutation of $\Fq$ can be represented by a polynomial of the form
\begin{equation}\label{eqn:CarlitzRep} 
P_{n}(x)=(\ldots
((c_0x+c_1)^{q-2}+c_2)^{q-2} \ldots +c_n)^{q-2}+c_{n+1},
\end{equation}
where $c_i \neq 0$, for $i = 0,2,\ldots,n$. (Note that $c_{1}c_{n+1}$ can be zero.) This representation is not unique
and $n$ is not necessarily minimal.  We recall that the {\it Carlitz rank} $Crk(f)$ of a 
permutation polynomial $f(x)$ over $\Fq$ is the smallest 
integer $n\ge 0$ satisfying $f(x)=P_{n}(x)$ for a permutation polynomial $P_{n}(x)$ of the
form (\ref{eqn:CarlitzRep}). The Carlitz rank was first introduced in \cite{AkCeMeTo09} and further studied in \cite{GoOsTo14, IsWinTop16}. For a survey see \cite{Topuzoglu}.

In $2009$, Aksoy et al.\ \cite{AkCeMeTo09} showed 
\begin{equation}\label{ineq:deg}
Crk(f)\ge q-\deg(f)-1 ~~~\emph{if}  ~\deg(f)\ge 2. 
\end{equation}
In \cite{GoOsTo14} Gomez-Perez et al.\ gave a similar bound for $Crk(f)$ in terms of the weight $w(f)$ of $f(x)$, that is  the number of its nonzero coefficients. If
$f(x)\neq a+bx^{q-2}$, for all $a,b \in \Fq$, $b\neq 0$, then
\begin{equation}\label{ineq:weight}
Crk(f)>\frac{q}{w(f)+2}-1~~~\emph{if}  ~\deg(f)\ge2.
\end{equation}
In this paper, we study the relationship between the Carlitz rank and the least index of a polynomial introduced in \cite{NidWin05,Wang} as follows.

Let $\ell$ be a positive divisor of
$q-1$ and $\xi$ a primitive element of $\Fq$. Then the set of nonzero $\ell$th powers
$$C_{0}=\left\{\xi^{j\ell} : j=0,1,...,\frac{q-1}{\ell}-1\right\}$$ is a subgroup of $\Fq^{*}=\Fq\setminus\{0\}$ of index $\ell$. The elements of the factor group $\Fq^{*}/C_{0}$ are the {\it cyclotomic cosets}

\begin{equation*}\label{eqn:C_i}
C_{i}=\xi^{i}C_{0}, ~~ i=0,1,...,\ell-1.
\end{equation*}

For any positive integer $r$ and any $a_0, a_1,..., a_{\ell-1} \in \Fq^{*}$, we define the {\it $r$-th order cyclotomic
mapping $f^{r}_{a_0,a_1,...,a_{\ell-1}}$ of index $\ell$} by
\begin{equation}\label{eqn:rth-cycMap}
f^{r}_{a_0,a_1,...,a_{\ell-1}}(x)=\left\{\begin{array}{cl} 0 & \mbox{if $x=0$},\\
                                            a_{i}x^{r} & \mbox{if $x\in C_i$, $0\le i \le \ell-1$}.\\
                                            \end{array}\right.
\end{equation}
For a polynomial $f(x)$ over $\F_q$ with $f(0)=0$ we denote by $Ind(f)$ the smallest index $\ell$ such that $f(x)$ can be represented in the form (\ref{eqn:rth-cycMap}).
The index was introduced in \cite{Wang} based on \cite{NidWin05} and further studied in \cite{muwawa,wawa,wa}.
The index of $f(x)$ with $f(0)\ne 0$ is defined as the index of $f(x)-f(0)$. Since $Crk(f+c)$ is the same as $Crk(f)$ for any $c \in \Fq$ we may restrict ourselves to the case $f(0)=0$. 


By \cite[Theorem~1]{Wang}, $f^{r}_{a_0,a_1,...,a_{\ell-1}}(x)$ is a permutation of $\F_q$ if and only if $\gcd(r,(q-1)/\ell)=1$ and 
$\{a_{0},a_{1}\xi^r,a_{2}\xi^{2r},\dots,a_{\ell-1}\xi^{(\ell-1)r}\}$ is a system of distinct representatives of $\Fq^{*}\setminus C_{0}$.
In particular, $f^{r}_{a_0,a_1,\ldots,a_{\ell-1}}(x)$ is a permutation if $\gcd(r,(q-1)/\ell)=1$ and  all $a_0,a_1,\ldots,a_{\ell-1}$ are in the same cyclotomic coset $C_i$ of order $\ell$.  


Now we define two further measures for the unpredictability of a polynomial.
The {\em linearity} ${\cal L}(f)$ of a polynomial $f(x)$ over $\F_q[x]$ with $f(0)=0$ is 
$${\cal L}(f)=\max_{a\in \F_q^{*}} |\{ c\in \F_q : f(c)=ac\}|$$
and the {\em invertibility} ${\cal I}(f)$ of $f(x)$  is
$${\cal I}(f)=\max_{c\in \Fq^{*}}\left|\left\{x\in \Fq^{*} : f(x)=\frac{c}{x}\right\}\right|.$$

For cryptographic applications we need unpredictable permutation polynomials, see for example \cite{sh}. In particular, 
${\cal L}(f)$ and ${\cal I}(f)$ must both be small, and $Crk(f)$, $\deg(f)$, and $w(f)$ must all be large. 
In Section \ref{Chap:MainResults} we prove a relation between $Crk(f)$ and $Ind(f)$ of the same flavour as (\ref{ineq:deg}) and (\ref{ineq:weight})   
provided that ${\cal L}(f)$ and ${\cal I}(f)$ are not large. 
We improve this result for large index $\ell$ in the special case when the coefficients $a_{0},a_{1},\ldots,a_{\ell-1}$ in (\ref{eqn:rth-cycMap}) 
are all in the same cyclotomic coset of index $\ell$. 

Moreover, in Section \ref{DiscLog} we provide an example of a permutation polynomial of small linearity, small invertibility, 
large degree, large weight, large index and large Carlitz rank. This polynomial represents (up to an additive constant $1$) the discrete logarithm of $\F_p$ for prime $p$.

\section{Preliminary results }

Let $f(x)$ be a permutation polynomial of $\F_q$. Then there exist
$\alpha,\beta,\gamma,\delta\in \F_q$ such that
\begin{equation}\label{rateq} f(c)=\frac{\alpha c+\beta}{\gamma c+\delta} \quad \mbox{for at least $q-Crk(f)$ different elements $c\in \F_q$},
\end{equation}
see \cite[Section~2]{Topuzoglu}. We may assume 
$$\alpha\delta \ne \beta\gamma$$ 
since otherwise $\dfrac{\alpha x+\beta}{\gamma x+\delta}$ is constant and thus $Crk(f)\ge q-1$.

Note that if $\beta =\gamma =0$, then $f(c)=\dfrac{\alpha}{\delta} c$ for at least $q-Crk(f)$ different elements $c \in \Fq$ and thus 
\begin{equation}\label{CarlitzRankL}
Crk(f)\ge q-{\cal L}(f)~~~~\mbox{if}~~\beta=\gamma=0. 
\end{equation}
If $\alpha=\delta=0$, then $f(c)=\dfrac{\beta}{\gamma c}$ for at least $q-Crk(f)$ different $c \in \Fq$ and thus
\begin{equation}\label{CarlitzRankI}
Crk(f)\ge q -{\cal I}(f)~~~~\mbox{if}~~\alpha=\delta=0.
\end{equation}

Now let $f(x)$ be an $r$-th order cyclotomic mapping permutation polynomial of index $\ell$ as defined in (\ref{eqn:rth-cycMap})
and assume that $\beta\ne 0$ or $\gamma\ne 0$ or $r\ge2$. Then
(\ref{rateq}) has at most $(r+1)\ell$ solutions and we get
\begin{equation}\label{rCrkIndex}
Crk(f)\ge q-(r+1)Ind(f), ~~~\mbox{$\beta\ne 0$ or $\gamma \ne 0$}.
\end{equation}
If $\alpha\ne 0$ or $\delta \ne 0$ or $r\le q-3$, then
$$a_{i}c^{r}=\frac{\alpha c +\beta}{\gamma c+ \delta}$$
or equivalently
$$\gamma c+\delta= a_{i}^{-1} (\alpha c^{q-r} +\beta c^{q-r-1})$$  
has at most $q-r$ solutions and thus 
\begin{equation}\label{rCrkIndex2}
Crk(f)\ge q- (q-r) Ind(f),~~~ \mbox{$\alpha\ne 0$ or $\gamma \ne 0$}.
\end{equation}
Collecting (\ref{CarlitzRankL}), (\ref{CarlitzRankI}), (\ref{rCrkIndex}), (\ref{rCrkIndex2}) we get 
\begin{equation*}
Crk(f)\ge q-\max\Big\{\min\{r+1,q-r\}Ind(f), {\cal L}(f), {\cal I}(f)\Big\}.
\end{equation*}

Our goal is to find a similar bound which does not depend on $r$.

\section{Main results}\label{Chap:MainResults}

\begin{thm}\label{thm:main} For any permutation polynomial $f(x)$ we have 
$$Crk(f)> q- \max \big\{3 Ind(f), (3q)^{1/2}, {\cal L} (f),{\cal I}(f) \big\}.$$
\end{thm}

\begin{proof} By (\ref{CarlitzRankL}) and (\ref{CarlitzRankI}) we may restrict ourselves to the case when neither $\beta=\gamma=0$ nor $\alpha=\delta= 0$. 

Put $\ell=Ind(f)$, that is, $f(x)$ is of the form (\ref{eqn:rth-cycMap}).

Let $N$ be the number of solutions $c\in \F_q$ of 
$$f(c)=\dfrac{\alpha c+\beta}{\gamma c+ \delta}.$$
Put $R(c)=\frac{\alpha c +\beta}{\gamma c+ \delta}$ and note that $N$ is the cardinality of the set
$${\cal N} =\big\{(\xi^i,y): i=0,1,\ldots,\ell-1, y\in  C_{0} \mbox{ with }a_i\xi^{ir} y^{r}=R(\xi^i y)\big \}.$$
Let $N_i$ be the contribution to $N$ for fixed $i=0,1,\ldots,\ell-1$. Then by the Cauchy-Schwarz inequality we have 
$$N^2=\left(\sum_{i=0}^{\ell -1} N_i\right)^{2}\le \ell M,$$ 
where
$$M= \sum_{i=0}^{\ell -1} N_i^{2}.$$
Now we have 
\begin{eqnarray*} M&=&\big|\{(\xi^i,y,z) : i=0,1,\ldots,\ell-1, y,z\in C_0\\ 
&& \quad \mbox{ with }  a_i\xi^{ir}y^r=R(\xi^iy)  \mbox{ and } a_i\xi^{ir}y^{r}z^{r}=R(\xi^iyz) \}\big|.\end{eqnarray*}
Then
$$R(\xi^i yz)=a_i\xi^{ir}y^{r}z^r=z^{r}R(\xi^iy),$$
that is,
\begin{equation*}
\frac{\alpha \xi^iyz + \beta}{\gamma \xi^iyz +\delta}=z^r \frac{\alpha \xi^iy +\beta}{\gamma \xi^iy+ \delta},
\end{equation*}
which implies
\begin{equation}\label{eqn:Long}
\alpha \gamma \xi^{2i}y^{2}z + (\alpha \delta z + \beta \gamma)\xi^iy+\beta \delta=  \alpha \gamma \xi^{2i}y^2 z^{r+1}+(\alpha \delta  + \beta \gamma z)\xi^iyz^r +\beta \delta z^r.
\end{equation}

First we consider the case $\alpha \gamma \ne 0$. We note that $z^{r+1}=z$ implies $z=1$ since $z \in C_{0}$ and $\gcd(r,\frac{q-1}{\ell})=1$. Therefore for the case $z=1$, (\ref{eqn:Long}) is true for all $N$ pairs $(\xi^i,y)\in {\cal N}$. 
For each $z\ne 1$ there are at most two solutions $\xi^iy$. Thus $$M\le N+2\left(\frac{q-1}{\ell}-1\right),$$ which implies $$N^2< \ell N+2q\le 3 \max \{\ell N, q\}.$$
Hence we get
\begin{equation}\label{eqn:max1}
N<\max \{3\ell, (3q)^{1/2}\}.
\end{equation}
Now we suppose that $\alpha \gamma = 0$. First consider the case $\alpha=0$. Note that here $\delta\ne 0$ and  $\beta\gamma \ne 0$.
Then we have  $\beta\gamma \xi^iy +\beta \delta =\beta \gamma \xi^iy  z^{r+1}  + \beta \delta z^{r}$. This equation has at most one solution $\xi^iy$ for each $z\ne 1$ and $N$ solutions if $z= 1$. Therefore
$$M\le N+\Big(\frac{q-1}{\ell}-1\Big),$$ which implies 
\begin{equation}\label{eqn:max2}
N< \max \{ 2 \ell, (2q)^{1/2}\}.
\end{equation} 
The case $\gamma=0$ is similar. 
Now the result follows by (\ref{rateq}), (\ref{eqn:max1}) and (\ref{eqn:max2}).
\end{proof}

Remark. 
\cite[Theorem~1]{NidWin05} implies that $\deg(f) \ge \dfrac{q-1}{Ind(f)}+1$ and the right hand side of (\ref{ineq:deg}) is at most $q-\dfrac{q-1}{Ind(f)}-2$. 
Hence Theorem \ref{thm:main}  improves (\ref{ineq:deg}) if $\max\{Ind(f), {\cal L}(f), {\cal I}(f)\}\le (q/3)^{1/2}$ and (\ref{ineq:weight}) 
if $\max\{Ind(f), {\cal L}(f), {\cal I}(f)\}\le 2q/9$.\\

We note that the following theorem gives a better bound whenever $Ind(f) > q^{1/2}$ provided that all coefficients $a_{i}$ are in the same cyclotomic coset. 

\begin{thm}\label{thm:samecoset}
Let $f(x)$ be a permutation polynomial of $\F_q$ of index $\ell$ of the form (\ref{eqn:rth-cycMap}) with all $a_i$ in the same coset of order $\ell$. Then 
$$Crk(f)\ge q-\max\{ 3 q^{1/2}, {\cal L} (f),{\cal I}(f) \}.$$
\end{thm}
\begin{proof} Since otherwise the result follows from Theorem \ref{thm:main} we may restrict ourselves to the case $Ind(f) > q^{1/2}$ 
as well as to the case that neither $\alpha=\delta=0$ nor $\beta=\gamma=0$ by (\ref{CarlitzRankL}) and (\ref{CarlitzRankI}).

Let $\chi$ be a character of order $\ell$, then $\chi(a_i)=\rho$ for all $i=0,1,...,\ell-1$. (Note that some of the $a_i$ can be the same and we can do this also
if $\ell \ge |C_0|=(q-1)/\ell$, that is, $\ell>q^{1/2}$.) Under this condition $f(x)$ is a
permutation and we can estimate the number $N$ of $c$ satisfying
\begin{equation}\label{eqn:thm2.1}
R(c)=\frac{\alpha c+\beta}{\gamma c+ \delta}=f(c).
\end{equation}
If (\ref{eqn:thm2.1}) is true, then we have $\chi\big(R(c)\big)=\chi(c^r)\rho$. Since
\begin{equation*}
\theta(c)=\frac{1}{\ell} \sum_{i=0}^{\ell-1} \Big(\chi(\alpha c +\beta) 
\overline{\chi}(\gamma c +\delta)\overline{\chi}(c^r)\rho^{-1}\Big)^i=\left\{\begin{array}{cl} 1, & \chi\big(R(c)\big)=\chi(c^r)\rho,\\
                                0, & \chi\big(R(c)\big)\neq \chi(c^r)\rho,
                                \end{array}\right. 
                                \end{equation*}
for any $c$ with $c(\alpha c+\beta)(\gamma c+\delta)\ne 0$,
we have $$N \le\sum_{c\in \Fq^{*}\setminus \{-\beta\alpha^{-1},-\delta\gamma^{-1}\}} \theta(c),$$
and get
\begin{eqnarray*}
N&<& \frac{1}{\ell} \sum_{i=0}^{\ell-1} \sum_{c\in \Fq} \Big(\chi(\alpha c +\beta) \overline{\chi}(\gamma c +\delta)\overline{\chi}(c^r)\rho^{-1}\Big)^i\\
&=&\frac{q}{\ell}+ \max_{i=1,2,\ldots,\ell-1} \left|\sum_{c\in \Fq} \Big(\chi(\alpha c +\beta) \overline{\chi}(\gamma c +\delta)\overline{\chi}(c^r)\Big)^{i}\right|
\le\frac{q}{\ell}+ 2q^{1/2} \le  3q^{1/2}
\end{eqnarray*}
by the Weil bound, see for example \cite[Theorem 5.41]{LidNid97}, and the 
result follows by (\ref{rateq}). 
\end {proof}

\section{The discrete logarithm of $\F_p$}\label{DiscLog}

In this section we consider the following permutation over $\F_p$ where $p>2$ is a prime which represents up to $+1$ the discrete logarithm of $\F_p$,
\begin{equation}\label{disclog} f(\xi^i)=i+1,\quad i=0,1,\ldots,p-2,\quad f(0)=0,
\end{equation}
where $\xi$ is again a primitive element of $\F_p$. 
We show that any polynomial $f(x)$ representing this permutation has large degree, weight, Carlitz rank and index as well as small linearity
and invertibility.
\begin{thm}
 The unique polynomial $f(x)\in \F_p[x]$ of degree at most $p-1$ defined by the property
 (\ref{disclog}) satisfies
 $$\deg(f)=w(f)=p-2,$$
 $$Ind(f)=p-1,\quad p>3,$$
 $${\cal L}(f)<(2(p-2))^{1/2}+1,$$
 $${\cal I}(f)<2(p-2)^{1/2}+1,$$
 and
 $$Crk(f)>p-2(p-2)^{1/2}-1.$$
\end{thm}
\begin{proof}
By \cite{muwh} we have
$$f(c)=\sum_{i=1}^{p-2} (\xi^{-i}-1)^{-1}c^i,\quad c\in \F_p,$$
and the first result follows.

Next we estimate the index $\ell$ of $f(x)$. Assume $\ell\le (p-1)/3$. Then there is some $a\in \F_p^*$ with 
$$f(\xi^{i\ell})=a\xi^{ir\ell},\quad i=0,1,\ldots,(p-1)/\ell -1\ge 2.$$
Taking $i=0,1,2$ we get $a=1$, 
$$\ell +1 =\xi^{r\ell}\quad \mbox{and}\quad 2\ell +1=\xi^{2r\ell}$$
which implies
$$(\xi^{r\ell}-1)^2=0.$$
Since $\gcd(r,(p-1)/\ell)=1$, we get $\ell=p-1$, a contradiction.
Finally we have to exclude $\ell=(p-1)/2$. Otherwise we have
$$(p+1)/2=f(\xi^{(p-1)/2})=f(-1)=(-1)^r$$
which is impossible since $1<(p+1)/2<p-1$ for $p>3$. Hence, 
$$Ind(f)=p-1,\quad p>3.$$

The proof of \cite[Theorem 8.2]{sh} can be easily adapted to estimate the number~$N$ of solutions of 
$$R(c)=\frac{\alpha c+\beta}{\gamma c+\delta}=f(c)$$
in the case $\gamma\ne 0$. The case $\gamma=0$ follows directly from \cite[Theorem 8.2]{sh} and gives the bound
$$N(N-1)\le 2(p-2)$$
which implies
$${\cal L}(f)<(2(p-2))^{1/2}+1$$
and 
\begin{equation}\label{e1} N<(2(p-2))^{1/2}+1,\quad \gamma=0.
\end{equation}

Now consider the set 
$$D=\{2\le a\le p-1 : f(c)=R(c) \mbox{ and }f(ac)=R(ac) \mbox{ for some $c$}\}$$
of size $|D|\le p-2$. There are $N(N-1)$ pairs $(c,ac)$, $a\ne 1$, with
\begin{equation}\label{a} f(c)=R(c)\quad\mbox{and}\quad f(ac)=R(ac)
\end{equation}
and for some $a\in D$ (\ref{a}) has at least $N(N-1)/|D|>(N-1)^2/(p-2)$ solutions~$c$. For this $a$ 
$$R(ac)=f(ac)=f(c)+d=R(c)+d$$
has either for $d=f(a)\not\in \{0,1\}$ or for $d=f(a)-1\not\in \{-1,0\}$ at least $(N-1)^2/2(p-2)$ solutions~$c$.
On the other hand we have
$$R(ac)=\frac{\alpha ac+\beta}{\gamma ac +\delta}=\frac{\alpha c+\beta}{\gamma c+\delta}+d=R(c)+d$$
which implies a quadratic equation for $c$ since $ad\gamma\ne 0$ with at most $2$ solutions and thus
\begin{equation}\label{gammanot0}(N-1)^2\le 4(p-2),\quad \gamma\not= 0.
\end{equation}
 The case $\alpha=\delta=0$ provides the bound on the invertibility ${\cal I}(f)$.
The bound on the Carlitz rank follows by (\ref{rateq}), (\ref{e1}) and (\ref{gammanot0}).
\end{proof}

Finally, we note that the results of this section can be easily extended to arbitrary finite fields using the approach of \cite{wi}.

\section*{Acknowledgement}

The authors are supported by the Austrian Science Fund FWF Projects F5504 and F5511-N26, respectively, 
which are  part of the Special Research Program "Quasi-Monte Carlo Methods: Theory and Applications". 
L.I. would like to express her sincere thanks for the hospitality during her visit to RICAM. 
The authors also like to thank Igor Shparlinski for pointing to an improvement of an earlier version of Theorem 1 as well as  
Alev Topuzo\u glu for suggesting earlier research on Carlitz rank which motivated this paper and for all of her kind help.


\begin{thebibliography}{99}


\bibitem{AkCeMeTo09} E. Aksoy, A.\c Ce\c smelio\u glu, W. Meidl, A. Topuzo\u glu,
\textit{On the Carlitz rank of a permutation polynomial},
\emph{Finite Fields Appl.} 15 (2009), 428--440.

\bibitem{Carlitz53} L. Carlitz,
\textit{Permutations in a finite field}, \emph{Proc. Amer. Math. Soc.} 4 (1953), 538.


\bibitem{GoOsTo14} D. Gomez-Perez, A. Ostafe, A. Topuzo\u glu,
\textit{On the Carlitz rank of permutations of $\Fq$ and pseudorandom sequences}, 
\emph{J. Complexity} 30 (2014), 279-289.

\bibitem{IsWinTop16} L. I\c s\i k, A. Topuzo\u glu, A. Winterhof,
\textit{Complete mappings and Carlitz rank},
\emph{Des. Codes Cryptogr.} (to appear).


\bibitem{LidNid97} R. Lidl, H. Niederreiter, 
Finite Fields. Second edition. Encyclopedia of Mathematics and its Applications, 20. Cambridge University Press, Cambridge, 1997.

\bibitem{muwawa} G. L. Mullen, D. Wan, Q. Wang, 
\textit{ Index bounds for value sets of polynomials over finite fields}, Applied algebra and number theory, 280--296, Cambridge Univ. Press, Cambridge, 2014.

\bibitem{muwh} G. L. Mullen, D. White, 
\textit{A polynomial representation for logarithms in $GF(q)$},
\emph{Acta Arith.} 47 (1986), 255--261.


\bibitem{NidWin05} H. Niederreiter, A. Winterhof,
\textit{Cyclotomic $\mathcal{R}$-orthomorphisms of finite fields},
\emph{Discrete Math.} 295 (2005), 161--171.

\bibitem{sh} I. Shparlinski, Cryptographic Applications of Analytic Number Theory. Complexity Lower Bounds and Pseudorandomness. 
Progress in Computer Science and Applied Logic, 22. Birkh\"auser Verlag, Basel, 2003.

\bibitem{Topuzoglu} A. Topuzo\u glu, \textit{Carlitz rank of permutations of
finite fields: a survey}, \emph{J. Symbolic Comput.}  64
(2014), 53--66.

\bibitem{wawa} D. Wan, Q. Wang, 
\textit{Index bounds for character sums of polynomials over finite fields},
\emph{Des. Codes Cryptogr.} 81 (2016), 459--468. 

\bibitem{Wang} Q. Wang,
\textit{Cyclotomic mapping permutation polynomials over finite fields.}
\emph{Sequences, Subsequences, and Consequences (International Workshop, SSC 2007}, Los Angeles, CA, USA, May 31 - June 2, 2007), 
Lecture Notes in Comput. Sci. 4893, 119--128, Springer, Berlin, 2007.

\bibitem{wa} Q. Wang, 
\textit{Cyclotomy and permutation polynomials of large indices},
\emph{Finite Fields Appl.} 22 (2013), 57--69. 

\bibitem{wi} A. Winterhof, \textit{Polynomial interpolation of the discrete logarithm}, \emph{Des. Codes Cryptogr.} 25 (2002), 63--72.
\end{thebibliography}
\end{document}